\documentclass[11pt]{amsart}
\usepackage{amssymb,amsthm,amsmath,amstext}
\usepackage{stmaryrd}
\usepackage{mathrsfs} 

\usepackage[left=1.5in,top=1.2in,right=1.5in,bottom=1.2in]{geometry}

\usepackage[all]{xy}

\theoremstyle{plain}
\newtheorem{theorem}{Theorem}[section]
\newtheorem*{theorem*}{Theorem}
\newtheorem{prop}[theorem]{Proposition}
\newtheorem{lemma}[theorem]{Lemma}
\newtheorem{cor}[theorem]{Corollary}

\newtheorem{theoremintro}{Theorem}

\theoremstyle{definition}

\newtheorem{example}[theorem]{Example}

\theoremstyle{remark}
\newtheorem{remark}[theorem]{Remark}

\newcommand{\Z}{\mathbb Z}
\newcommand{\C}{\mathbb C}
\renewcommand{\P}{\mathbb P}

\newcommand{\cC}{\mathcal{C}}
\newcommand{\cK}{\mathcal{K}}
\newcommand{\cPf}{\mathcal{P}\!\mathit{f}}
\newcommand{\Cliff}{\gamma}
\newcommand{\CH}{\mathrm{CH}}
\newcommand{\sign}{\mathrm{sign}}

\newcommand{\Mg}{\mathcal{M}}
\newcommand{\mar}{\mathrm{mar}}
\newcommand{\Kthree}{\mathcal{K}}
\newcommand{\Db}{\mathsf{D}^{\text{b}}}

\newcommand{\linedef}[1]{\textsl{#1}}
\newcommand{\bslash}{\smallsetminus}
\newcommand{\tensor}{\otimes}
\DeclareMathOperator{\Pic}{\mathrm{Pic}}
\newcommand{\dual}{^{\vee}}
\newcommand{\isom}{\cong}
\newcommand{\inv}{^{-1}}
\DeclareMathOperator{\disc}{\mathrm{disc}}
\newcommand{\exterior}{{\textstyle \bigwedge}}

\newcommand{\sheaf}[1]{\mathscr{#1}}
\newcommand{\OO}{\sheaf{O}}
\newcommand{\EE}{\sheaf{E}}

\usepackage[backref=page]{hyperref}
\usepackage{hyperref}
\hypersetup{pdftitle={Brill-Noether special cubic fourfolds of discriminant 14}}
\hypersetup{pdfauthor={Asher Auel}}
\hypersetup{colorlinks=true,linkcolor=blue,anchorcolor=blue,citecolor=blue}

\begin{document}

\title[Brill--Noether special cubic fourfolds]{Brill--Noether special cubic fourfolds \\of discriminant 14} 

\author{Asher Auel}
\address{Department of Mathematics\\%
Dartmouth College\\%
Kemeny Hall\\%
Hanover, NH 03755}
\email{asher.auel@dartmouth.edu}



\begin{abstract}
  We study the Brill--Noether theory of curves on polarized K3
  surfaces that are Hodge theoretically associated to cubic fourfolds
  of discriminant 14.  We prove that any smooth curve in the
  polarization class has maximal Clifford index and deduce that a
  smooth cubic fourfold contains disjoint planes if and only if admits
  a Brill--Noether special associated K3 surface of degree 14.  As an
  application, we prove that the complement of the pfaffian locus,
  inside the Noether--Lefschetz divisor $\cC_{14}$ in the moduli space
  of smooth cubic fourfolds, is contained in the irreducible locus of
  cubic fourfolds containing two disjoint planes.
\end{abstract}

\dedicatory{To Bill Fulton, on the occasion of his 80th birthday.}

\maketitle

\vspace*{-.4cm}

\section*{Introduction}
\label{sec:introduction}

Let $X$ be a \linedef{cubic fourfold}, i.e., a smooth cubic
hypersurface $X \subset \P^5$ over the complex numbers.  Determining
the rationality of $X$ is a classical question in algebraic geometry.
Some classes of rational cubic fourfolds have been described by Fano
\cite{fano:cubic} and Tregub \cite{tregub:rationality_cubic_fourfold},
\cite{tregub:remarks_cubic_fourfolds}.  Beauville and
Donagi~\cite{beauville_donagi} prove that \linedef{pfaffian} cubic
fourfolds, i.e., those defined by pfaffians of skew-symmetric
$6\times 6$ matrices of linear forms, are rational.
Hassett~\cite{hassett:special} describes, via lattice theory,
Noether--Lefschetz divisors $\cC_d$ in the moduli space $\cC$ of
smooth cubic fourfolds.  A parameter count shows that $\cC_{14}$ is
the closure of the locus $\cPf$ of pfaffian cubic fourfolds; Hodge
theory shows (see \cite[\S3~Prop.~2]{voisin:cubic_fourfolds}) that
$\cC_8$ is the locus of cubic fourfolds containing a plane.
Hassett~\cite{hassett:some} identifies countably many divisors of
$\cC_8$ consisting of rational cubic fourfolds.  Recently, Addington,
Hassett, Tschinkel, and V\'arilly-Alvarado \cite{AHTVA} identify
countably many divisors of $\cC_{18}$ consisting of rational cubic
fourfolds, and Russo and Staglian\`o \cite{russo_stagliano},
\cite{russo_stagliano:42} have shown that the very general cubic
fourfolds in $\cC_{26}$, $\cC_{38}$, and $\cC_{42}$ are rational.
Nevertheless, it is expected that the very general cubic fourfold (as
well as the very general cubic fourfold containing a plane) is not
rational.

\smallskip

Short of a pfaffian presentation, how can one tell if a given cubic
fourfold is pfaffian?
Beauville~\cite{beauville:determinantal_hypersurfaces} provides a
homological criterion for a cubic hypersurface to be pfaffian, which
for cubic fourfolds is equivalent to containing a quintic del Pezzo
surface, but it is not clear how to translate this criterion into
Hodge theory.  More generally, how can one understand the complement
$\cC_{14} \bslash \cPf$ of the pfaffian locus?  Such questions are
implicit in \cite{ABBV:pfaffian} and
\cite{tregub:rationality_cubic_fourfold}, where cubic fourfolds with
certain numerical properties are shown to be outside or inside,
respectively, the pfaffian locus.  In particular, Tregub studies the
locus $\cC_\Pi$ of cubic fourfolds that contain two disjoint planes,
showing that this locus is irreducible of codimension 2 in $\cC$, and
that the general member does not contain a smooth quartic rational
normal scroll nor a quintic del Pezzo surface, hence cannot be
pfaffian.  Our main result is that this is essentially all of the
complement of the pfaffian locus.

\begin{theoremintro}
\label{thm:main}
The complement of the pfaffian locus $\cPf$, inside the
Noether--Lefschetz divisor $\cC_{14}$ of the moduli space of cubic
fourfolds, is contained in the irreducible locus $\cC_\Pi$ of cubic fourfolds
containing two disjoint planes.

In other words, any $X \in \cC_{14}$ is
pfaffian or contains two disjoint planes (or both).
\end{theoremintro}

The proof combines several ingredients revolving around the
Brill--Noether theory of special divisors on curves in K3 surfaces of
degree 14.  We use the determination, due to
Mukai~\cite{mukai:curves_grassmannians}, \cite{mukai:Fano_threefolds}
of the smooth projective curves $C$ of genus 8 that are linear
sections of the grassmannian $G(2,6) \subset \P^{14}$.  This turns out
to be equivalent to $C$ lacking a $g^2_7$, equivalently, that $C$ is
Brill--Noether general.  We also use a modified conjecture of Harris
and Mumford, as proved by Green and
Lazarsfeld~\cite{green_lazarsfeld:divisors_curves_K3_surface}, as well
as the generalization due to
Lelli-Chiesa~\cite{lelli-chiesa:genealized_Lazarsfeld-Mukai_bundles},
on line bundles on K3 surfaces computing the Clifford indices of
smooth curves in a given linear system.  We also need the earlier work
of Saint-Donat~\cite{saint-donat:projective_models_K3} and
Reid~\cite{reid:hyperelliptic_linear_systems_K3},~\cite{reid:special_linear_systems_curves_K3},
on hyperelliptic and trigonal linear systems on K3 surfaces, as well
as useful refinements due to Knutsen~\cite{knutsen:kth_order},
\cite{knutsen:smooth_curves_K3_surfaces},~\cite{knutsen:gonality_Clifford_index}
of the original result by Green and Lazarsfeld.  Combining these
results with lattice theory computations for cubic fourfolds and their
associated K3 surfaces, as developed by Hassett~\cite{hassett:some},
\cite{hassett:special}, we prove that the Clifford index of curves in
the polarization class of any K3 surface of degree 14 associated to
$X$ must take the maximal value 3 (see Theorem~\ref{thm:clifford=3}),
putting strong constraints on the geometry of cubic fourfolds in terms
of the Brill--Noether theory of their associated K3 surfaces.

More generally, one might call a cubic fourfold $X$
\linedef{Brill--Noether special} if $X$ has an associated K3 surface
$S$ that is Brill--Noether special in the sense of
Mukai~\cite[Def.~3.8]{mukai:Fano_threefolds}, a condition implying
that $S$ has an ample divisor such that the general curve in its
linear system is Brill--Noether special, see \S\ref{subsec:BN-K3}.
Then our main result can be summarized by saying that a special cubic
fourfold of discriminant 14 is Brill--Noether special if and only if
it contains two disjoint planes.  It would be interesting to study the
Brill--Noether special loci in other divisors $\cC_d$ of special cubic
fourfolds, for example, in $\cC_{26}$, $\cC_{38}$, and $\cC_{42}$. In
the context of discriminant $26$, Farkas and Verra~\cite{farkas_verra}
also appeal to the Brill--Noether theory of some associated K3 surfaces.

Our result, and more generally the ability to detect a pfaffian cubic
fourfold via Hodge theory, has two immediate applications.  First, we
obtain a new explicit proof that every cubic fourfold in $\cC_{14}$ is
rational: Beauville and Donagi \cite{beauville_donagi} prove that any
pfaffian cubic fourfold is rational, and by a much more classical
construction going back to Fano, every cubic fourfold containing
disjoint planes is rational; this covers all cubic fourfolds in
$\cC_{14}$.  This rationality result was initially obtained by
Bolognesi, Russo, and Staglian\`o \cite{bolognesi_russo} using a much
more classical approach involving one apparent double point surfaces,
though this has been recently subsumed by the path-breaking work on
the deformation invariance of rationality by Kontsevich and Tschinkel
\cite{kontsevich_tschinkel}.  Second, we prove the existence of
nonempty irreducible components of $\cPf \cap \Pi$, which are
necessarily of codimension $\geq 3$ in $\cC$.  This immediately
implies that the pfaffian locus is not Zariski open in $\cC_{14}$.
While this result was initially obtained in the course of
conversations with M.~Bolognesi and F.~Russo based on the computer
algebra calculations of G.~Staglian\`o and earlier drafts of our
respective papers, the proof presented in \S\ref{notopen} does not
require any explicit computer algebra computations (as opposed to the
proof in \cite{bolognesi_russo}).  However, it still seems plausible
that the pfaffian locus is open inside the moduli space of marked
cubic fourfolds of discriminant 14.

The author is indebted to Y.~Tschinkel and F.~Bogomolov for
providing a stimulating work environment at the Courant Institute of
Mathematical Sciences, where this project started in May 2013, and to
B.~Hassett, who first suggested the possibility of investigating the
Brill--Noether theory of curves on K3 surfaces in the context of cubic
fourfolds. 
The author also thanks M.~Bolognesi and F.~Russo for animated and
productive conversations during the preparation of this manuscript in
March 2015, while we were exchanging our respective drafts.  We are
grateful to M.~Hoff, D.~Jensen, A.~L.~Knutsen, and M.~Lelli-Chiesa for
detailed explanations of various aspects of their work; to
N.~Addington, T.~Johnsen, A.~Kumar, R.~Lazarsfeld, and H.~Nuer for
helpful conversations; and to the anonymous referee for very
constructive comments on the manuscript.  The author was partially
supported by NSF grant DMS-0903039 and an NSA Young Investigator
Grant.

\section{Brill--Noether theory for polarized K3 surfaces of degree 14}
\label{sec:Special_divisors}

All varieties are assumed to be over the complex numbers and all K3
surfaces are assumed to be smooth and projective.

\subsection{Grassmannians and curves of genus 8}

Let $G(2,6) \subset \P^{14}$ be the grassmannian of $2$-planes in a
$6$-dimensional vector space, embedded in $\P^{14}$ via the Pl\"ucker
embedding.  It was classically known that a general flag of linear
subspaces $P \subset Q$ of dimension 6 and 7 in $\P^{14}$ cut from
$G(2,6)$ a K3 surface of degree 14 containing a canonical curve $C$ of
genus 8. 

Recall that a $g^r_d$ on a smooth projective curve $C$ is a line
bundle $A$ of degree $d$ with $h^0(C,A) \geq r+1$; it is
\linedef{complete} if $h^0(C,A)=r+1$.

\begin{theorem}[{Mukai~\cite{mukai:curves_grassmannians}}]
\label{thm:mukai_g27}
A smooth projective curve $C$ of genus 8 is a linear
section of the grassmannian $G(2,6) \subset \P^{14}$ if and only if
$C$ has no $g^2_7$.
\end{theorem}

The Brill--Noether theorem states that when
$\rho(g,r,d) = g-(r+1)(g-d+r)$ is negative, the general curve of genus
$g$ has no $g^r_d$.  A curve supporting such a $g^r_d$ is called
\linedef{Brill--Noether special}.  A curve not supporting any $g^r_d$
whenever $\rho(g,r,d)<0$ is called \linedef{Brill--Noether general}.
When $\rho(g,r,d) = -1$, Eisenbud and
Harris~\cite{eisenbud_harris:irreducibility} proved that the locus of
curves, in the moduli space $\Mg_g$ of curves of genus $g$, that
support such a $g^r_d$, is irreducible of codimension~1.  In
particular, the locus of curves of genus 8 having a $g^2_7$ is of
codimension 1 in $\Mg_8$.

The \linedef{Clifford index} of a line bundle $A$ on a smooth
projective curve $C$ is the integer
$$
\Cliff(A) = \deg(A) - 2\, r(A),
$$
where $r(A) = h^0(C,A)-1$ is the \linedef{rank} of $A$.  The Clifford
index of $C$ is
$$
\Cliff(C) = \min \{\, \Cliff(A) \; : \; h^0(C,A) \geq 2
~\text{and}~h^1(C,A) \geq 2 \,\}
$$
and a line bundle $A$ on $C$ is said to compute the Clifford index of $C$ if
$\Cliff(A)=\Cliff(C)$. Clifford's theorem states that $\Cliff(C) \geq
0$ with equality if and only if $C$ is hyperelliptic; similarly
$\Cliff(C)=1$ if and only if $C$ is trigonal or a smooth plane
quintic.  At the other end, $\Cliff(C) \leq \lfloor (g-1)/2 \rfloor$
with equality whenever $C$ is Brill--Noether general.

Up to taking the adjoint line bundle $\omega_C \tensor A\dual$, which
has the same Clifford index, we can always assume that nontrivial
special divisors $g^r_d$ satisfy $1 \leq r \leq \lfloor (g-1)/2
\rfloor$ and $2 \leq d \leq g-1$.  For $g=8$, we list them for the
convenience of the reader:
$$
\renewcommand{\arraystretch}{1.2}
\begin{array}{||c||c|c||c|c||c|c|c||c|c|c||}
\hline\hline
\gamma &
\multicolumn{2}{|c||}{3} &
\multicolumn{2}{|c||}{2} &
\multicolumn{3}{|c||}{1} &
\multicolumn{3}{|c||}{0} \\\hline 
g^r_d & g^1_5 & g^2_7 & g^1_4 & g^2_6 & g^1_3 & g^2_5 & g^3_7 & g^1_2&g^2_4&g^3_6\\
\hline
\rho&0&-1&-2&-4&-4&-7&-8&-6&-10&-12\\\hline\hline
\end{array}
$$
In genus 8, the Brill--Noether special locus is controlled by the
existence of a $g^2_7$.

\begin{lemma}
\label{lem:BN-special_g27}
A smooth projective curve $C$ of genus 8 is Brill--Noether special if
and only if it has a complete $g^2_7$.
\end{lemma}
\begin{proof}
First note that if a curve has a complete $g^r_d$, then it has a
complete $g^k_d$ for all $k$ between $d-g$ and $r$.  Hence if $C$ has
a $g^2_7$ then it has a complete $g^2_7$.
We can argue by the Clifford index.  In Clifford index~3, the only
special divisor is a $g^2_7$.  For Clifford index 2, we use the facts
that any genus 8 curve with a $g^2_6$ has a $g^1_4$ and any genus 8
curve with a $g^1_4$ has a $g^2_7$, see
\cite[Lemmas~3.4,~3.8]{mukai:curves_grassmannians}.  In Clifford index~1, any genus 8 curve is trigonal, so taking twice the $g^1_3$ and
adding a base point will result in a $g^2_7$.  Finally, in Clifford
index 0, the curve is hyperelliptic, so taking thrice the~$g^1_2$ and
adding a base point will result in a $g^2_7$.
\end{proof}

\subsection{Brill--Noether theory for polarized K3 surfaces}
\label{subsec:BN-K3}

A \linedef{polarized K3 surface} $(S,H)$ of degree $d$ is a smooth
projective K3 surface $S$ together with a primitive ample line bundle
$H$ of self-intersection $d \geq 2$.  If $C \subset S$ is a smooth
irreducible curve in the linear system $|H|$, so that $|H|$ is base
point free (i.e., $H$ is globally generated), then
$d = 2g-2$, where $g$ is the genus of $C$.  Following
Mukai~\cite[Def.~3.8]{mukai:Fano_threefolds}, we say that a polarized
K3 surface $(S,H)$ of degree $2g-2$ is \linedef{Brill--Noether
  general} if $h^0(S,H')\, h^0(S,H'') < h^0(S,H)=g+1$ for any
nontrivial decomposition $H=H'\tensor H''$.  Otherwise, we say
\linedef{Brill--Noether special}.  If a smooth irreducible curve
$C \in |H|$ is Brill--Noether general then it follows that $(S,H)$ is
Brill--Noether general, cf.\ \cite[Rem.~10.2]{johnsen_knutsen:K3}.
While the converse is an open question in general, for low degrees it
was checked by Mukai, using a case-by-case analysis.

\begin{theorem}
\label{thm:BNgeneral}
A polarized K3 surface $(S,H)$, with $H$ globally generated of degree
$\leq 18$ or $22$, is Brill--Noether general if and only if some smooth
irreducible $C \in |H|$ is Brill--Noether general.
\end{theorem}

Of course, we are mainly interested in the degree 14 case, where the
results assembled below will suffice to prove the theorem.

The existence of special divisors on curves in a K3 surface was
considered by Saint-Donat \cite{saint-donat:projective_models_K3} and
Reid \cite{reid:hyperelliptic_linear_systems_K3},
\cite{reid:special_linear_systems_curves_K3}.  Harris and Mumford
conjectured that the gonality of a curve should be constant in a base
point free linear system on a K3 surface.  A counterexample was found
by Donagi and Morrison \cite{donagi_morrison} (in fact, this turned
out to be the unique counterexample, cf.\ \cite{ciliberto_pareschi},
\cite{knutsen:two_conjectures_curves_K3_surfaces}) and the conjecture
was modified by Green~\cite[Conj.~5.8]{green:koszul_1} to one about
the constancy of the Clifford index in a linear system.  In a similar
spirit, one is interested in the question of when a given $g^r_d$ on a
curve in a K3 surface is the restriction of a line bundle from the K3.
The conjecture of Green was proved in a celebrated paper by Green and
Lazarsfeld.

\begin{theorem}[Green--Lazarsfeld~\cite{green_lazarsfeld:divisors_curves_K3_surface}]
\label{thm:green_laz}
Let $S$ be a K3 surface and $C \subset S$ a smooth irreducible curve
of genus $g \geq 2$.  Then
$
\Cliff(C') = \Cliff(C)
$
for every smooth curve $C' \in |C|$.  Furthermore, if $\Cliff(C) <
\lfloor (g-1)/2 \rfloor$ then there exists a line bundle $L$ on $S$
whose restriction to any $C' \in |C|$ computes the Clifford index of
$C'$.
\end{theorem}

We can thus define the \linedef{Clifford index} $\Cliff(S,H)$ of a
polarized K3 surface $(S,H)$ with $H$ globally generated to be the
Clifford index of any smooth irreducible curve $C \in |H|$, which is
well-defined by Theorem~\ref{thm:green_laz}.

In the case where $(S,H)$ has degree 14, so that a smooth curve
$C \in |H|$ has genus 8, we have that $\Cliff(S,H) \leq 3$.
If $(S,H)$ is Brill--Noether general, then by
Theorem~\ref{thm:BNgeneral}, some smooth curve $C \in |H|$ is
Brill--Noether general (hence has maximal Clifford index), so that
$\Cliff(S,H)=3$.
When $(S,H)$ is Brill--Noether special and ${\Cliff(S,H) < 3}$, the
result of Green and Lazarsfeld allows us to find a line bundle on $S$
whose restriction to $C \in |H|$ computes the Clifford index.  In
fact, already for $\Cliff(S,H) \leq 1$, results of
Saint-Donat~\cite[Thm.~5.2]{saint-donat:projective_models_K3} and
Reid~\cite[Thm.~1]{reid:hyperelliptic_linear_systems_K3} ensure that
these line bundles can be chosen to be elliptic pencils, see
\S\ref{subsec:lattice_polarized_K3} for details.  Finally, when
$(S,H)$ is Brill--Noether special and $\Cliff(S,H)=3$, we would like
to know if a $g^2_7$ on a smooth curve $C \in |H|$ is the restriction of a
line bundle on~$S$.  Since a $g^2_7$ has the generic Clifford
index, we cannot appeal to the result of Green and Lazarsfeld.  This
situation, of Clifford general but not Brill--Noether general
polarized K3 surfaces, is discussed more generally in
\cite[\S10.2]{johnsen_knutsen:K3}.

To this end, we have the following much more powerful result of
Lelli--Chiesa, concerning when a specific $g^r_d$ on a curve $C
\subset S$ lying in a K3 surface is the restriction of a line bundle
on $S$.

\begin{theorem}[Lelli-Chiesa \cite{lelli-chiesa:genealized_Lazarsfeld-Mukai_bundles}]
\label{thm:lelli}
Let $S$ be a K3 surface and $C \subset S$ a smooth irreducible curve
of genus $g \geq 2$ that is neither hyperelliptic nor trigonal.  Let
$A$ be a complete $g^r_d$ such that $r>1$, $d \leq g-1$, $\rho(g,r,d)
< 0$, and $\Cliff(A)=\Cliff(C)$.  Assume that there is no irreducible
genus 1 curve $E \subset S$ such that $E.C=4$ and no
irreducible genus 2 curve $B \subset S$ such that $B.C=6$.  Then $A$
is the restriction of a globally generated line bundle $L$ on $S$.
\end{theorem}

This result comes from an in-depth study of generalized
Lazarsfeld--Mukai bundles extending the original strategy of
\cite{green_lazarsfeld:divisors_curves_K3_surface}.

\begin{remark}
\label{rem:hyp}
According to
\cite[Thm.~4.2\textit{ff}.]{lelli-chiesa:genealized_Lazarsfeld-Mukai_bundles},
the hypothesis on curves of genus 1 and 2 is completely satisfied as
long as $\Cliff(C) > 2$; otherwise, there is a list of seven
exceptional cases when $\Cliff(C)=2$.  
We also remark that, according to the construction in the proof of
\cite[Thm.~4.2]{lelli-chiesa:genealized_Lazarsfeld-Mukai_bundles} (see
also \cite[Lemma~3.3]{lelli-chiesa:stability_rank3} and
\cite[Lemma~3.1]{green_lazarsfeld:divisors_curves_K3_surface}), the
line bundle $L$ can be chosen to be globally generated, though this is
not mentioned in the statement of the main theorem in
\cite{lelli-chiesa:genealized_Lazarsfeld-Mukai_bundles}.
\end{remark}

When a K3 surface has Picard rank one,
Lazarsfeld~\cite{lazarsfeld:Brill-Noether_without_degenerations} has
shown that the general curve in the linear system of the polarization
class is Brill--Noether general.  Hence Brill--Noether special K3
surfaces have higher Picard rank.

\subsection{Brill--Noether special K3 surfaces via lattice-polarizations}
\label{subsec:lattice_polarized_K3}

Let $\Sigma$ be an even nondegenerate lattice of signature
$(1,\rho-1)$ with a distinguished class $H$ of even norm $d > 0$.  A
\linedef{$\Sigma$-polarized} K3 surface is a polarized K3 surface
$(S,H)$ of degree $d$ together with a primitive isometric embedding
$\Sigma \hookrightarrow \Pic(S)$ preserving~$H$.  For a general
discussion of lattice-polarized K3 surfaces and their moduli, see
\cite{dolgachev:lattice_polarized_K3}. 
In particular, there exists a quasi-projective coarse moduli space
$\Kthree_\Sigma$ of dimension $20-\rho$ and a forgetful morphism
$\Kthree_\Sigma \to \Kthree_d$ to the moduli space of polarized K3
surfaces of degree~$d$. 
The main result of this section is the following characterization of
Brill--Noether special K3 surfaces of degree 14 via lattice
polarizations.  The same result is obtained by Greer, Li, and
Tian~\cite{greer_li_tian} using a different calculation.

\begin{theorem}
\label{thm:classify_K3}
If a polarized K3 surface $(S,H)$, with $H$ globally generated of
degree 14, is Brill--Noether special then it admits a lattice
polarization for one of the five rank 2 lattice appearing in
Table~\ref{tab:BN-special_K3} and $\gamma(S,H)$ is bounded above by
the corresponding value of $\gamma$ on the table.  
\end{theorem}

\begin{table}
$$
\begin{array}{||c||c||c||c||c||}
\hline\hline
\gamma=0
&
\gamma=1
&
\multicolumn{2}{c||}{\gamma=2}
&
\gamma=3\\\hline
\begin{minipage}{2.2cm}
\vspace{1mm}
$
\begin{array}{c|cc}
\multicolumn{1}{c}{}  & H  & E \\\cline{2-3}
H & 14 & 2 \\ 
E & 2  & 0
\end{array}
$\\[1.5mm]
\end{minipage}
&
\begin{minipage}{2.2cm}
\vspace{1mm}
$
\begin{array}{c|cc}
\multicolumn{1}{c}{}  & H  & E \\\cline{2-3}
H & 14 & 3 \\ 
E & 3  & 0
\end{array}
$
\\[1.5mm]
\end{minipage}
&
\begin{minipage}{2.2cm}
\vspace{1mm}
$
\begin{array}{c|cc}
\multicolumn{1}{c}{}  & H  & E \\\cline{2-3}
H & 14 & 4 \\ 
E & 4  & 0
\end{array}
$\\[1.5mm]
\end{minipage}
&
\begin{minipage}{2.2cm}
\vspace{1mm}
$
\begin{array}{c|cc}
\multicolumn{1}{c}{}  & H  & L \\\cline{2-3}
H & 14 & 6 \\ 
L & 6  & 2
\end{array}
$\\[1.5mm]
\end{minipage}
&
\begin{minipage}{2.2cm}
\vspace{1mm}
$
\begin{array}{c|cc}
\multicolumn{1}{c}{}  & H  & L \\\cline{2-3}
H & 14 & 7 \\ 
L & 7  & 2
\end{array}
$\\[1.5mm]
\end{minipage}
\\\hline
d_S=-4
&
d_S=-9
&
d_S=-16
&
d_S=-8
&
d_S=-21\\
d_S^0=-14
&
d_S^0=-14\cdot 9
&
d_S^0=-14\cdot 4
&
d_S^0=-14\cdot 2
&
d_S^0=-6\\\hline
(b,c)=(6,8)
&
(b,c)=(5,6)
&
(b,c)=(2,2)
&
(b,c)=(4,4)
&
(b,c)=(7,12)
\\\hline\hline
\end{array}
$$
\caption{Lattices embedded in Brill--Noether special K3 surfaces of
degree 14 and Clifford index $\gamma$. Here, $d$ and $d_0$ denote the
discriminants of the lattice and of $\langle H\rangle^\perp$,
respectively. The pair $(b,c)$ refers to the unique rank 3 cubic
fourfold lattice, whose associated lattice is the given one, normalized as in Proposition~\ref{prop:deg_14_lattice}.}
\label{tab:BN-special_K3}
\end{table}

To round out the classification, we remark that elliptic K3 surfaces
$(S,H)$ of degree 14 with a section are also Brill--Noether special.
These admit a lattice polarization with a class $E$ such that $E^2=0$
and $E.H=1$.  However, in this case, the linear system $|H|$ contains
a nontrivial fixed component in its base locus, and in particular,
there is no Clifford index defined.  Taken together with the five
lattices listed in Table~\ref{tab:BN-special_K3}, this shows that the
Brill--Noether special locus in $\Kthree_{14}$ is the union of six
Noether--Lefschetz divisors.

Before the proof of the Theorem~\ref{thm:classify_K3}, we need some
lemmas on elliptic pencils on K3 surfaces, which are mostly contained
in the work of Saint-Donat~\cite{saint-donat:projective_models_K3} and
Knutsen~\cite{knutsen:smooth_curves_K3_surfaces},
\cite{knutsen:gonality_Clifford_index}.  By an \linedef{elliptic
  pencil} we mean a line bundle $E$ on a K3 surface $S$ such that the
generic member of the linear system $|E|$ is a smooth genus one curve.
A result of
Saint-Donat~\cite[Prop.~2.6(ii)]{saint-donat:projective_models_K3}
says that if $E$ is generated by global sections and $E^2=0$, then $E$
is a multiple of an elliptic pencil.  If $E$ is an elliptic pencil
then $E$ is primitive in $\Pic(S)$ (cf. \cite[Ch.~2,
~Remark~3.13(i)]{huybrechts:K3_book}), $E^2=0$, $h^0(S,E)=2$, and
$h^1(S,E)=0$.

\begin{lemma}
\label{lem:compute_gonality}
Let $(S,H)$ be a polarized K3 surface of degree $2g-2 \geq 2$, let
$C \in |H|$ be a smooth irreducible curve, and let $E$ be a globally
generated line bundle on $S$ with $E^2=0$ and $E.C=d < 2g-2$.  Then
$E|_C$ is a $g^1_d$ if and only if $E$ is an elliptic pencil such that
$h^1(S,E(-C))=0$.
\end{lemma}
\begin{proof}
First remark that since $H$ is globally generated and $(E-C).C = d -
(2g-2) <0$ by hypothesis, we get that $h^0(S,E(-C))=0$, cf.\
\cite[Proof~of~Prop.~2.1]{knutsen:gonality_Clifford_index}.

Now, assume that $E$ is an elliptic pencil and that $h^1(S,E(-C))=0$.
Then the long exact sequence in cohomology associated to the exact
sequence of sheaves
$$
0 \to E(-C) \to E \to E|_C \to 0
$$
together with the fact that $h^0(S,E)=2$, implies that
$h^0(C,E|_C)=2$.  Since $\deg(E|_C) = E.C = d$, we have that $E|_C$ is
a $g^1_d$.

Now assume that $E|_C$ is a $g^1_d$.  By
Saint-Donat~\cite[Prop.~2.6(ii)]{saint-donat:projective_models_K3},
$E=F^{\tensor k}$ for an elliptic pencil $F$ and some $k\geq 1$
dividing $d$.  Again considering the same long exact sequence as
above, the last terms, when rewritten using Serre duality and the fact
that $H^0(S,E\dual)=0$ since $E$ is effective, read
$$
H^1(S,E|_C) \to H^0(S,E(-C)\dual) \to 0.
$$
By Riemann--Roch on $C$, we have $h^1(S,E|_C) = h^1(C,E|_C) =
2-(d-g+1)$, hence $h^0(S,E(-C)\dual) \leq 2-(d-g+1)$.  By
Riemann--Roch on $S$, we have
$$
h^0(S,E(-C)\dual)-h^1(E(-C)) = 2 + \frac{1}{2}(E-C)^2 = 2 +
\frac{1}{2}(-2d+2g-2) = 2 - (d-g+1),
$$
using Serre duality and the fact that $H^0(S,E(-C))=0$, hence
$h^0(S,E(-C)\dual) \geq 2-(d-g+1)$ and $h^1(S,E(-C))=0$.  However, the
beginning terms of the long exact sequence read
$$
0 \to H^0(S,E) \to H^0(S,E|_C) \to H^1(E(-C))
$$
implying that $h^0(S,E)=h^0(C,E|_C)=2$ (since $E|_C$ is a $g^1_e$).  But
$h^0(S,E)=k+1$ and thus we conclude that $k=1$, i.e., $E$ is an
elliptic pencil.
\end{proof}

\begin{proof}[Proof of Theorem~\ref{thm:classify_K3}]
Let $C \subset S$ be a smooth irreducible curve (of genus 8) in
the linear system of $H$.
We argue by the Clifford index of $(S,H)$, equivalently, of $C$.

If $\Cliff(C)=0$, i.e., $C$ is hyperelliptic by Clifford's Theorem,
then by Saint-Donat \cite[Thm.~5.2]{saint-donat:projective_models_K3}
(cf.\ Reid \cite[Prop.~3.1]{reid:hyperelliptic_linear_systems_K3}),
the $g^1_2$ on $C$ is the restriction of an elliptic pencil $E$ such
that $E.H=2$.  

If $\Cliff(C)=1$, i.e., $C$ is trigonal, then by
Reid~\cite[Thm.~1]{reid:special_linear_systems_curves_K3} (cf.\
\cite[Thm.~7.2]{saint-donat:projective_models_K3}), after verifying $8
> \frac{1}{4}3^2+3+2$, the $g^1_3$ on $C$ is the restriction of an
elliptic pencil $E$ such that $E.H=3$.  

In these first two cases, the sublattice of $\Pic(S)$ generated by $H$
and $E$ is primitive.  Indeed, if not, then this sublattice admits a
finite index overlattice contained in $\Pic(S)$.  However, using the
correspondence between finite index overlattices and isotropic subgroups of
the discriminant form (cf., Nikulin~\cite[\S1.4]{nikulin}), we find
that, in this case, the only finite index overlattice would admit
a class $F \in \Pic(S)$, where $e F = E$, for $e=2$ or $3$,
respectively.  However, as $E$ is an elliptic pencil on $S$, it is a
primitive class in $\Pic(S)$, hence no such overlattice exists.

If $\Cliff(C)=2$, then by
Green--Lazarsfeld~\cite{green_lazarsfeld:divisors_curves_K3_surface}
(since the generic value of the Clifford index is 3, see
Theorem~\ref{thm:green_laz}), there is a line bundle $L$ on $S$ such
that $L|_C$ is a $g^1_4$ or a $g^2_6$.  Then $L.H = \deg(L|_C) = 4$ or
$6$, respectively.  Furthermore, by a result of
Knutsen~\cite[Lemma~8.3]{knutsen:kth_order}, we can choose $L$
satisfying
$$
0 \leq L^2 \leq 4 \quad\text{and}\quad 2L^2 \leq L.H
\quad\text{and}\quad 2 = L.H - L^2 - 2
$$
with $L^2=4$ or $2L^2 = L.H$ if and only if $H=2L$.  However, since
$14$ is squarefree, $H=2L$ is impossible, hence the only possibilities
are that $L^2=0$ and $L.H=4$, $L^2=0$ and $L.H=6$, or $L^2=2$ and
${L.H=6}$.  As a consequence of Martens' proof
\cite{martens:curves_on_K3_surfaces} of the main result of
\cite{green_lazarsfeld:divisors_curves_K3_surface} (cf.\ proof of
\cite[Lemma~8.3]{knutsen:kth_order}), we can also choose $L$ generated
by global sections and with $h^1(S,L(-C))=0$.  Suggestively, in the
two former cases, we denote $L$ by $E$.

We now argue that the case $E^2=0$ and $E.H=6$ is impossible.  First
assume that $E$ is an elliptic pencil.
Lemma~\ref{lem:compute_gonality} then implies that $E|_C$ is a
$g^1_6$, contradicting the assumption that it is a $g^2_6$.  Hence $E$
cannot be an elliptic pencil.  Thus by the result of Saint-Donat
mentioned above, $E = k F$ for $k=2,3,6$ and an elliptic pencil $F$.
The case $k=6$ is impossible, since $F^2=0$ and $F.H=1$ contradicts
the ampleness of $H$.  For $k=2,3$, we have $F^2=0$ and
$F.H=6/k \leq 3$, so that results of
Saint-Donat~\cite[Prop.~5.2,~7.15]{saint-donat:projective_models_K3}
imply that $F|_C$ is a $g^1_{6/k}$, contradicting the fact that
$\gamma(C)=2$.

In the remaining two cases, we argue that the sublattice of $\Pic(S)$
generated by $H$ and $E$ (resp.\ $H$ and $L$) is primitive.  As
before, we appeal to the correspondence between finite index overlattices
and isotropic subgroups of the discriminant form (cf.,
Nikulin~\cite[\S1.4]{nikulin}).  In the case $E^2=0$ and $E.H=4$, the
only finite overlattice would contain a class dividing $E$, however
since $E|_C$ must be a $g^1_4$, then by
Lemma~\ref{lem:compute_gonality}, $E$ is an elliptic pencil and is
thus a primitive class in $\Pic(S)$.  Hence, the sublattice generated
by $H$ and $E$ is primitive.  In the case $L^2=2$ and $L.H=6$, the
only finite index overlattice would contain a class
$F \in \Pic(S)$ such that $2F = H-L$, however, such $F$ would then
satisfy $F^2 = (H-L)^2/4 = 1$, which is impossible since $\Pic(S)$ is
an even lattice.  Hence, the sublattice generated by $H$ and $L$ is
primitive.

Finally, assume that $\Cliff(C) = 3$.  Then all the hypotheses of the
results of
Lelli-Chiesa~\cite{lelli-chiesa:genealized_Lazarsfeld-Mukai_bundles}
(see Theorem~\ref{thm:lelli}) are satisfied, hence there exists a line
bundle $L$ on $S$ such that $L|_C$ is a $g^2_7$.  In particular, $L.C
= \deg(L|_C) = 7$.  As before, by Remark~\ref{rem:hyp}, $L$ can be
chosen to be globally generated, so that $2n=L^2 \geq 0$.
Furthermore, by \cite[Prop.~10.5]{johnsen_knutsen:K3}, we can choose
$L$ so that $L^2=2$.  The sublattice of $\Pic(S)$ generated by $H$ and
$L$ is then primitive since its discriminant is squarefree.

Thus in each case, the polarized K3 surface $(S,H)$ has a
lattice-polarization with respect to one of the lattices on
Table~\ref{tab:BN-special_K3}.  
\end{proof}

\begin{remark}
\label{rem:g15}
Every smooth curve $C$ of genus 8 contains a finite number of $g^1_5$
divisors.  If $\gamma(C)=3$ and $C$ lies on a K3 surface $S$ with a
primitive degree 14 polarization $H$, then it could happen that none
of the $g^1_5$ divisors are the restriction of a line bundle from $S$
(e.g., the corresponding Lazarsfeld--Mukai bundles are simple).
However, if a $g^1_5$ is the restriction of a line bundle on $S$,
then arguing as in the proof of Theorem~\ref{thm:classify_K3}, one can
verify that the Picard lattice of $S$ 
admits a primitive sublattice 
generated by $H$ and $E$,
where $E$ is an elliptic pencil such that $H.E=5$ and $E|_C$ is the
$g^1_5$.
\end{remark}

\section{Lattice polarized cubic fourfolds}
\label{sec:Some_lattice_polarized_cubic_fourfolds}

Let $X$ be a smooth cubic fourfold and let $A(X)$ denote the lattice
of codimension 2 algebraic cycles $\CH^2(X)$ with its usual
intersection form.  Then via the cycle class map, $A(X)$ is isomorphic
to $H^4(X,\Z) \cap H^{2,2}(X)$ by the validity of the integral Hodge
conjecture for cubic fourfolds proved by
Voisin~\cite{voisin:aspects_Hodge_conjecture}.

Given a positive definite lattice $\Lambda$ containing a distinguished
element $h^2$ of norm $3$, a \linedef{$\Lambda$-polarized} cubic
fourfold is a cubic fourfold $X$ together with the data of a primitive
isometric embedding $\Lambda \hookrightarrow A(X)$ preserving $h^2$.
The main results of Looijenga~\cite{looijenga:period_cubic_fourfold}
and Laza~\cite{laza:period_cubic_fourfold} on the description of the
period map for cubic fourfolds imply that smooth $\Lambda$-polarized
cubic fourfolds exist if and only if~$\Lambda$ admits a primitive
embedding into
$H^4(X,\Z) = \langle 1 \rangle^{\oplus 21} \oplus \langle -1
\rangle^{\oplus 2}$ and $\Lambda$ contains no short roots (i.e.,
elements $v \in \Lambda$ with norm 2 such that $v.h^2=0$) nor long
roots (i.e., elements $v \in \Lambda$ with norm 6 such that $v.h^2=0$
and $v.\langle h^2 \rangle^{\perp} \subset 3\Z$ where we compute
$\langle h^2 \rangle^{\perp} \subset H^4(X,\Z)$).  We call any such
lattice $\Lambda$ a \linedef{cubic fourfold lattice}.  We remark that
the conditions defining short and long roots can be checked completely
within the lattice $\Lambda$: for short roots, this is clear; for long
roots $v \in \Lambda$, the conditions are equivalent to $v.v=6$,
$v.h^2=0$, and $v \pm h^2$ is divisible by 3 in $\Lambda$.

For a cubic fourfold lattice $\Lambda$ of rank $\rho$, an adaptation
of the argument of Hassett~\cite[Thm.~3.1.2]{hassett:special} (see
also \cite[\S2.3]{hassett:cubic_survey}) proves that the moduli space
$\cC_{\Lambda}$ of $\Lambda$-polarized cubic fourfolds is a
quasi-projective variety of dimension $21-\rho$.  There is a forgetful
map $\cC_{\Lambda} \to \cC$, whose image we denote by
$\cC_{[\Lambda]}$.  In other words, $\cC_{[\Lambda]} \subset \cC$ is
the locus of cubic fourfolds $X$ such that $A(X)$ admits a primitive
isometric embedding of $\Lambda$ preserving $h^2$.  We remark that the
forgetful map $\cC_\Lambda \to \cC_{[\Lambda]}$ is generically finite
to one, and whose degree depends on the number of automorphisms of
$\Lambda$ fixing~$h^2$.

The possible rank 2 cubic fourfold lattices were classified by
Hassett~\cite{hassett:special}; such a lattice $K_d$ is uniquely
determined by its discriminant, which can be any number $d > 6$ such
that $d \equiv 0,2 \pmod 6$.  Then $\cC_{K_d}$ coincides with the
moduli space $\cC_d^\mar$ of \linedef{marked} special cubic fourfolds
of discriminant $d$ considered by Hassett
\cite[\S5.2]{hassett:special} and $\cC_{[K_d]}$ coincides with the
Noether--Lefschetz divisor $\cC_d \subset \cC$.
For cubic fourfold lattices $\Lambda$ of rank 3, the loci
$\cC_{[\Lambda]}$ were considered in \cite{addington_thomas},
\cite{ABBV:pfaffian}, \cite{bolognesi_russo}, \cite{galluzzi},
\cite{tregub:rationality_cubic_fourfold},
\cite{tregub:remarks_cubic_fourfolds}.

Given a primitive embedding $\Lambda \hookrightarrow \Lambda'$ of
cubic fourfold lattices preserving $h^2$, there is an induced morphism
$\cC_{\Lambda'} \to \cC_{\Lambda}$ and an inclusion of
subvarieties $\cC_{[\Lambda']} \subset \cC_{[\Lambda]}$. 
In particular, we have that $\cC_{[\Lambda]} \subset \cC_d$ whenever
$\Lambda$ admits a primitive embedding of~$K_d$ preserving $h^2$.

When $\Lambda=\Pi$ is the lattice with Gram matrix
\begin{equation}
\label{eq:Pi}
\begin{array}{c|ccc}
\multicolumn{1}{c}{}   &h^2& T & P \\\cline{2-4}
h^2 & 3 & 4 & 1 \\
T   & 4 & 10 & -1 \\
P   & 1 & -1 & 3
\end{array}
\quad
\isom
\quad
\begin{array}{c|ccc}
\multicolumn{1}{c}{}   &h^2& P & P' \\\cline{2-4}
h^2 & 3 & 1 & 1 \\
P   & 1 & 3 & 0 \\
P'  & 1 & 0 & 3
\end{array}
\end{equation}
with the isomorphism defined by $T = 2h^2-P-P'$, then $\cC_{[\Pi]}$ is
one of the most well-studied codimension $2$ loci in the moduli space
of cubic fourfolds, cf.\ \cite{fano:cubic},
\cite{tregub:rationality_cubic_fourfold},
\cite[\S3,~App.]{voisin:cubic_fourfolds}.

\begin{prop}
\label{prop:Pi}
The subvariety $\cC_{[\Pi]}\subset \cC$ is an irreducible component of
$\cC_8 \cap \cC_{14}$ and coincides with the locus of cubic fourfolds
that contain disjoint planes.  
\end{prop}
\begin{proof}
The proof of the first statement is in \cite[Thm.~4]{ABBV:pfaffian},
cf.\ \cite{fano:cubic}, \cite{tregub:rationality_cubic_fourfold}.  The
existence of two disjoint planes follows from the proof given in
Voisin~\cite[\S3,~App.,~Prop.]{voisin:cubic_fourfolds} and the
refinement due to Hassett~\cite[\S3]{hassett:special}.  
\end{proof}

Fix an \linedef{admissible discriminant} $d>6$, i.e., such that
$d \equiv 0,2 \bmod 6$ and such that $4 \nmid d$, $9 \nmid d$, and
$p \nmid d$ for any odd prime $p \equiv 2 \bmod 3$.
Hassett~\cite[\S5]{hassett:special} proves that for any cubic fourfold
$X$ with a marking of discriminant $d$, the orthogonal complement
$K_d^\perp$ of $K_d$ inside $H^4(X,\Z)$ is Hodge isometric to a twist
$\Pic(S)_0(-1)$ of the primitive cohomology lattice of a polarized K3
surface $(S,H)$ of degree $d$, and that such a Hodge-theoretic
association gives rise to a choice of open immersion
$\cC_{K_d} = \cC_d^\mar \hookrightarrow \Kthree_d$ of moduli spaces
(cf.\ \cite[Corollary~5.2.4]{hassett:special}).  The choice of such an
open immersion is determined by an isomorphism between the
discriminant forms of the abstract lattices $K_d^\perp$ and
$\Pic(S)_0(-1)$, modulo scaling by $\{\pm 1\}$; there are $2^{r-1}$
such choices, where $r$ is the number of distinct odd primes dividing
$d$, see \cite[Corollary~5.2.4]{hassett:special},
\cite[Proposition~26]{hassett:cubic_survey}.

Now, given a cubic fourfold lattice $\Lambda$ and a fixed primitive
embedding $K_d \hookrightarrow \Lambda$ preserving $h^2$, we are
interested in generalizing this open immersion to $\Lambda$-polarized
cubic fourfolds.  We can do this explicitly in the case of interest to
us, namely when $d=14$ and the rank of $\Lambda$ is 3, due to the
following lemma.

\begin{lemma}
  \label{lem:sigma}
  Let $\Lambda$ be a rank 3 cubic fourfold lattice with a fixed
  primitive embedding $K_{14} \hookrightarrow \Lambda$ preserving $h^2$.
  Then, up to isometry, there is a unique rank 2 even indefinite
  lattice $\sigma(\Lambda)$ with discriminant $-d(\Lambda)$, a
  distinguished class $H$ of norm $d$, and such that the orthogonal
  complement of $K_{14}$ in $\Lambda$ is isometric (up to twist) with the
  orthogonal complement of $H$ in $\sigma(\Lambda)$.
\end{lemma}
\begin{proof}
As the sublattice
$K_{14} = \langle h^2,T \rangle \subset \Lambda$ is primitive, there exists a
class $J \in \Lambda$ and integers $a,b,c$ such that
$$
\begin{array}{c|ccc}
\multicolumn{1}{c}{}   &h^2& T & J \\\cline{2-4}
h^2 & 3 & 4 & a \\
T   & 4 & 10 & b \\
J   & a & b & c
\end{array}
$$
By translating $J$ to $J-a(T-h^2)$, we can assume that $a=0$.
Directly computing the determinant of this Gram matrix, we then find
that $d(\Lambda) = -3b^2 + 14c \equiv (5b)^2$ is a square modulo 14.
Let $0 \leq \alpha \leq 7$ be such that $\alpha^2 \equiv d(\Lambda)$
modulo 14.  Then we can write $d(\Lambda) = \alpha^2 - 14 \beta$ for
some integer $\beta$.  Now we argue that $\beta$ is even.  Since $J$
is orthogonal to $h^2$ and $\langle h^2 \rangle^\perp$ is an even
lattice, we have that $J^2=c$ must be even.  Thus
$d(\Lambda) \equiv 0, 1 \pmod 4$.  From the equation
$d(\Lambda) = \alpha^2-14\beta$ we see that $d(\Lambda)$ and $\alpha$
have the same parity, and by looking modulo 4, we finally find that
$\beta$ must be even.

We now define $\sigma(\Lambda)$ to be the rank 2 lattice $\langle H, L
\rangle$ with Gram matrix
$$
\begin{array}{c|cc}
\multicolumn{1}{c}{} & H  & L \\\cline{2-3}
H & 14 & \alpha \\ 
L & \alpha  & \beta 
\end{array}
$$
Then $d(\sigma(\Lambda)) = 14\beta - \alpha^2 = -d(\Lambda)$ and hence
$\sigma(\Lambda)$ is an indefinite even lattice since $d(\Lambda) > 0$
and $\beta$ is even.

We now directly calculate that the orthogonal complement of $K_{14}$
in $\Lambda$ is generated by $(4b h^2 - 3b T + 14 J)/\!\gcd(b,14)$ and
that the orthogonal complement of $H$ in $\sigma(\Lambda)$ is
generated by $(\alpha H - 14 L)/\!\gcd(\alpha,14)$.  Computing the
self-intersections of these generators yields
$14 d(\Lambda)/\!\gcd(b,14)^2$ and
$-14 d(\sigma(\Lambda))/\!\gcd(\alpha,14)^2$, respectively.  Noting
that $\alpha \equiv \pm 5b$ modulo 14, we have that
$\gcd(b,14) = \gcd(\alpha,14)$, which proves the claim about the
isometry of orthogonal complements.

Finally, we remark that $\sigma(\Lambda)$ is unique up to isometry
with these properties.  Indeed, given any rank 2 even indefinite
lattice with Gram matrix as above, after a translation and a possible
reflection, we can always choose $0 \leq \alpha \leq 7$.  But then
$\alpha,\beta$ are uniquely determined by the equation
$14\beta - \alpha^2 = -d(\Lambda)$.  So $\sigma(\Lambda)$ is unique up
to isometry.
\end{proof}

The proof of Lemma~\ref{lem:sigma} provides an algorithm, given the
Gram matrix of $\Lambda$, to calculate a Gram matrix of
$\sigma(\Lambda)$.  As an example, we calculate that the Gram matrix
of $\sigma(\Pi)$ is
\begin{equation}
\label{eq:hPi}
\begin{array}{c|cc}
\multicolumn{1}{c}{}  & H  & E \\\cline{2-3}
H & 14 & 7\\ 
E & 7 & 2 
\end{array}
\end{equation}
where $\Pi$ is the lattice in \eqref{eq:Pi}, with fixed primitive
embedding $K_{14} = \langle h^2, T \rangle \hookrightarrow \Pi$.

Now, for any rank 3 cubic fourfold lattice $\Lambda$ with a fixed
choice of primitive embedding $K_{14} \hookrightarrow \Lambda$ as in
Lemma~\ref{lem:sigma}, consider the moduli space
$\Kthree_{\sigma(\Lambda)}$ of $\sigma(\Lambda)$-polarized K3 surfaces
and the forgetful morphism
$\Kthree_{\sigma(\Lambda)} \to \Kthree_{14}$, whose image is a divisor
$\Kthree_{[\sigma(\Lambda)]} \subset \Kthree_{14}$.  For any
$\Lambda$-polarized cubic fourfold $X$, the fixed primitive embedding
$K_{14} \hookrightarrow \Lambda$ determines a discriminant $14$
marking of $X$, which induces an associated polarized K3 surface
$(S,H)$ of discriminant 14 admitting a $\sigma(\Lambda)$-polarization
by Lemma~\ref{lem:sigma}.  We recall that for discriminant $14$,
there is a unique choice of open immersion
$\cC_{K_{14}} \hookrightarrow \Kthree_{14}$, see \cite[\S6]{hassett:special}.  Then following
Hassett~\cite[\S5.2]{hassett:special}, we have the following.

\begin{prop}
\label{prop:marked_embedding}
Let $\Lambda$ be a rank 3 cubic fourfold lattice with a fixed
primitive embedding $K_{14} \hookrightarrow \Lambda$ preserving $h^2$.
Then there exists an open immersion
$\cC_\Lambda \hookrightarrow \Kthree_{\sigma(\Lambda)}$ of moduli
spaces and a commutative diagram
$$
\xymatrix@R=12pt@C=16pt{
\cC_{K_{14}} \ar@{^{(}->}[r] & \Kthree_{14}\\
\cC_\Lambda \ar[u] \ar@{^{(}->}[r] & \Kthree_{\sigma(\Lambda)} \ar[u]
}
$$
where the vertical arrows are the forgetful maps and the top
horizontal arrow is the (unique choice of) open immersion constructed
by Hassett.
\end{prop}

\section{Cubic fourfold lattice normal forms}
\label{sec:Fano_variety_of_lines}

This section is devoted to establishing normal forms for cubic fourfold
lattices of rank 3 with a discriminant 14 marking and their associated
K3 surface Picard lattices.  

\begin{prop}
\label{prop:deg_14_lattice}
Let $\Lambda$ be a rank 3 cubic fourfold lattice with a fixed
primitive embedding of $K_{14}$ preserving $h^2$.
Then there exists a basis $h^2, T, J$ of $\Lambda$ with respect to
which $\Lambda$ has Gram matrix
\begin{equation}
\label{eq:normal_form}
\begin{array}{c|ccc}
\multicolumn{1}{c}{}   &h^2& T & J \\\cline{2-4}
h^2 & 3 & 4 & 0 \\
T   & 4 & 10 & b \\
J   & 0 & b & c
\end{array}
\end{equation}
for some integers $0 \leq b \leq 7$ and $c > \max(2, 3b^2/14)$ even.
\end{prop}
\begin{proof}
Just as in the proof of Lemma~\ref{lem:sigma},  we can choose the
primitive sublattice $K_{14} = \langle h^2,T \rangle$, and then there
exists a class $J \in \Lambda$ and integers $b,c$ such that that Gram
matrix of $\Lambda$ has the shape \eqref{eq:normal_form}.
Since $(3T-4h^2).h^2=0$, we can further translate $J$ to
$\pm J-m(3T-4h^2)$, which preserves $h^2.J=0$ and allows us modify $b$
modulo $14 = (3T-4h^2).T$ and up to sign, so we can choose
representatives $0 \leq b \leq 7$.  Being a primitive sublattice of
$A(X)$, we know that $\Lambda$ is positive definite, hence its
discriminant $-3b^2+14c$ must be positive, which forces $c > 3b^2/14$.
Already in the proof of Lemma~\ref{lem:sigma}, we saw that $c$ must
be even; also $c$ must be greater than 2, since
$\langle h^2 \rangle^\perp$ is an even lattice with no vectors of norm 2.  Note that a similar normal form analysis is carried out
in \cite[\S2]{ABBV:pfaffian}, \cite[Lemma~4.2]{addington_thomas}.
\end{proof}

One application of the normal form in
Proposition~\ref{prop:deg_14_lattice} is that the lattice
$\sigma(\Lambda)$ can be even more explicitly computed from
$\Lambda$.  Given $(b,c)$ that determine $\Lambda$, we compute that
$\sigma(\Lambda)$ has Gram matrix
$$
\begin{array}{c|cc}
\multicolumn{1}{c}{}  & H  & L \\\cline{2-3}
H & 14 & 2b \\ 
L & 2b  & \frac{b^2}{2}-c
\end{array}
\qquad\text{or}\qquad
\begin{array}{c|cc}
\multicolumn{1}{c}{}  & H  & L \\\cline{2-3}
H & 14 & 7-2b\\ 
L & 7-2b & \frac{b^2-4b+7}{2}-c 
\end{array}
$$
depending on whether $b$ is even or odd, respectively.  

A consequence of this calculation is that $\sigma(\Lambda)$ together
with $H$ determines the pair $(b,c)$, and hence $\Lambda$ together
with the fixed primitively embedded $K_{14}$ up to isomorphism.  In
the last line of Table~\ref{tab:BN-special_K3}, we have recorded, by
listing the pair $(b,c)$, the unique rank 3 cubic fourfold lattices
$\Lambda$ with a primitive embedding $K_{14} \hookrightarrow \Lambda$
whose associated lattice is the given $\sigma(\Lambda)$.

Finally, we are in a position to deduce the following.

\begin{prop}
\label{prop:inPi}
If a smooth cubic fourfold $X$ has an associated K3 surface $(S,H)$ of
degree 14 that is Brill--Noether special and with Clifford index 3,
then $X \in \cC_{[\Pi]}$. 
\end{prop}
\begin{proof}
  The cubic fourfold $X$, together with the discriminant 14 marking
  $K_{14} \hookrightarrow A(X)$ whose associated K3 surface is
  $(S,H)$, determines a point on the moduli space $\cC_{K_{14}}$.
  Under the open immersion $\cC_{K_{14}} \hookrightarrow \Kthree_{14}$
  constructed by Hassett, this point maps to $(S,H)$.  Since $(S,H)$
  is Brill--Noether special with Clifford index 3,
  Theorem~\ref{thm:classify_K3} implies that $S$ admits a
  $\Sigma$-polarization where $\Sigma$ is the lattice in the
  $\gamma=3$ column of Table~\ref{tab:BN-special_K3}.  Hence $(S,H)$
  determines a point on the moduli space $\Kthree_{\Sigma}$.  Via the
  Hodge isometry $K_d^\perp \isom \Pic(S)_0(-1)$, we lift
  $\Sigma \cap \Pic(S)_0(-1)$ to a primitive rank 3 lattice
  $\Lambda \subset A(X)$; this is nothing but the saturation of the sum of $K_d$
  and the image of $\Sigma \cap \Pic(S)_0(-1)$ in $A(X)$.  A
  calculation of this saturation shows that, in fact, $\Lambda \isom \Pi$.  Recalling
  that $\Sigma = \sigma(\Pi)$ by \eqref{eq:hPi}, we thus have that the
  moduli point of $(S,H)$ in $\Kthree_{\sigma(\Pi)}$ is in the
  image of the open immersion
  $\cC_{\Pi} \hookrightarrow \Kthree_{\sigma(\Pi)}$, which finishes
  the proof by the commutativity of the diagram in
  Proposition~\ref{prop:marked_embedding}.
\end{proof}

In fact, later on in Theorem~\ref{thm:clifford=3}, we will show that
for any cubic fourfold $X$ with a discriminant 14 marking, the
polarized K3 surface $(S,H)$ of degree 14 Hodge theoretically
associated to $X$ always has Clifford index 3.

\bigskip

We end this section with some lattice computations that will be useful later on.
Let $\Lambda=(\Z^n,b)$ be an integral nondegenerate lattice with
bilinear form $b : \Z^n \times \Z^n \to \Z$ having Gram matrix $B$ and
discriminant $d(\Lambda)=\det(B)$.  Then with respect to the dual
standard basis, the dual lattice $\Lambda\dual=(\Z^n,b\dual)$ can be
considered as a bilinear form
$b\dual : \Z^n \times \Z^n \to \frac{1}{d}\Z$ having Gram matrix
$B\inv$.  The canonical isometric embedding $\Lambda \to \Lambda\dual$
is then identified with the matrix multiplication map
$B : (\Z^n,b) \to (\Z^n,b\dual)$.  In particular, the discriminant
group $\Lambda\dual/\Lambda$ can be identified with the cokernel of
the matrix $B$, which aids in explicit computations.


We state some useful, if not easy, necessary conditions for a lattice
to occur as the intersection lattice of a smooth cubic fourfold.

\begin{lemma}
\label{lem:easy_test}
No lattice of the following type can arise as the intersection lattice
$A(X)$ of a smooth cubic fourfold $X$:
\begin{enumerate}
\item \label{lem:easy_test_uni}
 A unimodular lattice.

\item \label{lem:easy_test_p}
 A lattice with odd rank $\rho \leq 11$ and discriminant a prime
$p \equiv \rho \bmod 4$.

\item \label{lem:easy_test_2odd}
 A lattice with odd rank $\rho \leq 11$ and discriminant exactly
divisible by 2.
\end{enumerate}
\end{lemma}
\begin{proof}
For \eqref{lem:easy_test_uni}, it is a consequence of the classification
of unimodular lattices of small rank that every unimodular lattice of
rank $\leq 22$ has short roots, hence cannot arise from a smooth cubic
fourfold.

We recall the notion of the discriminant form $q_A : A\dual/A \to
\Z/2\Z$ of $A=A(X)$, as well as the modulo 8 signature $\sign\, q_A$
considered in \cite[\S1]{nikulin}.  We remark that, in the notation of
\cite[Proposition~1.8.1]{nikulin}, for any odd prime $p$, we have that
$\sign\, q_\theta^p(p) \equiv 1-p \bmod 4$ and $\sign\,
q_\theta^p(p^2) \equiv 0 \bmod 8$ for any nonsquare class $\theta$
modulo $p$.  For \eqref{lem:easy_test_p}, by
\cite[Theorem~1.10.1]{nikulin}, 
for a lattice $A$ of odd rank $\leq 11$ to
be the intersection lattice of a smooth cubic fourfolds $X$, it is
necessary that the signature satisfy $\sign\, q_A \equiv 11-\rho \bmod
4$.  If $A$ has discriminant $p$, then $\sign\, q_A \equiv 1-p \bmod
4$, hence we must have, $p \equiv \rho - 2 \bmod 4$.  This is
impossible if $p \equiv \rho \bmod 4$.

For \eqref{lem:easy_test_2odd}, if 2 strictly divides $\disc(A)$, then
$q_A = q_\theta^2(2) + q(\text{odd})$, where $q(\text{odd})$ means a
finite quadratic form on a group of odd order.  By
\cite[Prop.~1.11.2*]{nikulin},
$\sign\, q_\theta^2(2) \not\equiv 0 \bmod 2$ while
$\sign\, q(\text{odd}) \equiv 0 \bmod 2$.  Hence no such cubic
fourfold $X$ exists.
\end{proof}

\section{Clifford index bounds for cubic fourfolds}
\label{sec:Pfaffian_cubic_fourfolds}

In this section, we recall the constructions of pfaffian cubic
fourfolds by Beauville and Donagi \cite{beauville_donagi} and
Brill--Noether general K3 surfaces of degree 14 by
Mukai~\cite[Thms.~3.9,~4.7]{mukai:Fano_threefolds}, working up to a
proof of our main results.  Throughout, denote by $\P(W)$ the
projective space of lines in a vector space $W$.

Let $V$ be a $\C$-vector space of dimension $6$ and consider the
subvarieties $G$ and $\Delta$ of $\P(\exterior^2 V)$ of tensors of
rank $2$ and $\leq 4$, respectively.  Then $G$ coincides with the
image of the Pl\"ucker embedding $G(2,V) \hookrightarrow
\P(\exterior^2 V)$, hence has dimension $8$ and degree $14$ by the
Schubert calculus, see \cite{fulton:young_tableaux}.  Also, $\Delta$ coincides with the vanishing locus
of the pfaffian map $\mathrm{pf} : \exterior^2 V \to \exterior^6 V$,
hence is a hypersurface of degree $3$.  We have that $G$ is the
singular locus of $\Delta$ and that $\Delta$ coincides with the secant
variety of $G$, see \cite[Rem.~1.5]{mukai:curves_grassmannians}.
Similarly, we define $G\dual \subset \Delta\dual \subset P(\exterior^2
V\dual)$.  Here, $\P(\exterior^2 V\dual)$ is the space of alternating
bilinear forms on $V$ up to homothety, and $\Delta\dual$ is the
subvariety of degenerate forms.

If $L \subset \P(\exterior^2 V)$ is a linear subspace of dimension 8
intersecting $G$ transversally then $S = L \cap G \subset L \isom
\P^8$ is the projective model of a smooth Brill--Noether general
polarized K3 surface $(S,H)$ of degree 14, see
\cite[Thm.~3.9]{mukai:Fano_threefolds}.  Conversely, if $(S,H)$ is a
Brill--Noether general polarized K3 surface of degree 14, then $S$ has
a \linedef{rigid} vector bundle $E$, unique up to isomorphism, such
that $E$ is stable of rank 2 with $\det E \isom H$ and
$\chi(S,E)=h^0(S,E)=6$, see \cite[Thm.~4.5]{mukai:Fano_threefolds}.
In particular, the evaluation morphism $H^0(S,E)\tensor \OO_S \to E$
is surjective, hence there is a grassmannian embedding $\Phi_E : S \to
G(2,H^0(S,E)\dual)$ taking $x \mapsto E_x\dual$.  Here, we think of
$E_x\dual \subset H^0(S,E)\dual$ as a 2-dimensional subspace dual to
the quotient map $H^0(S,E) \to E_x$ defined by the evaluation
morphism.  We have that $E = \Phi_E^*\EE$, where $\EE$ is the
tautological rank 2 vector (sub)bundle on $G(2,H^0(S,E)\dual)$.
Composing with the Pl\"ucker embedding, we have an embedding $S \to
\P(\exterior^2 H^0(S,E)\dual)$.  On the other hand, the exterior
square $\exterior^2 H^0(S,E) \tensor \OO_S \to \exterior^2 E$ of the
evaluation morphism defines a linear map $\lambda : \exterior^2
H^0(S,E) \to H^0(S,\exterior^2 E) \isom H^0(S,H)$.  As $\lambda$ is
surjective, we arrive at a commutative square of morphisms
$$
\xymatrix@R=23pt{
S \ar[d]_(.43){\Phi_H} \ar[r]^(.32){\Phi_E} & G(2,H^0(S,E)\dual)
\ar[d]^(.43){\text{\tiny Pl\"ucker}}\\ 
\P(H^0(S,H)\dual) \ar[r]^(.46)\mu & \P(\exterior^2 H^0(S,E)\dual)
}
$$
where $\mu$ is the linear embedding defined by $\lambda$.  A result of
Mukai is that this square is cartesian, see
\cite[Thm.~4.7]{mukai:Fano_threefolds}, hence $S$ can be written as an
intersection $S = L \cap G \subset \P(\exterior^2 V)$ in our previous
notation, where $L = \P(H^0(S,H)\dual)$, $V = H^0(S,E)\dual$, and
$G=G(2,V)$.  In conclusion, a polarized K3 surface $(S,H)$ of degree
14 is Brill--Noether general if and only if its projective model is a
transversal intersection $S = G \cap L \subset \P(\exterior^2 V)$ and
any such $(S,H)$ has a rigid vector bundle of rank 2, i.e., a rank 2
stable vector bundle $E$ with $\det(E) \isom H$ and $\chi(E)=6$.  By
\cite[Thm.~3.3(2)]{mukai:curves_K3_surfaces_Fano_3-folds},
Brill--Noether general polarized K3 surfaces $S = G \cap L$ and $S'=G
\cap L'$ are projectively equivalent if and only if $L$ and $L'$ are
equivalent under the action of $GL(V)$.  

Still letting $L \subset \P(\exterior^2 V)$ be a linear subspace of
dimension $8$, if the projective dual linear subspace
$L^\perp \subset \P(\exterior^2 V\dual)$ of dimension $5$ is not
contained in $\Delta\dual$, then
$X = L^\perp \cap \Delta\dual \subset L^\perp \isom \P^5$ is a
pfaffian cubic fourfold, see \cite[\S2]{beauville_donagi}.
Conversely, writing $L^\perp=\P(W)$ for a subspace
$W \subset \exterior^2 V\dual$, then $L^\perp$ gives rise to a global
section of the vector bundle $W\dual \tensor \OO_G(1)$, whose zero
locus is precisely $S$.

\begin{prop}
\label{prop:HPD}
Let $V$ be a vector space of dimension 6 and let
$L \subset \P(\exterior^2 V)$ be a linear subspace of dimension $8$.
Assume that $S=L \cap G$ has dimension 2 and that
$X=L^\perp \cap \Delta\dual$ has dimension 4.  If $X$ is smooth then
$S$ is smooth and in this case there is a semiorthogonal decomposition
$$
\Db(X) = \langle \mathcal{A}_X, \OO_X, \OO_X(1), \OO_X(2) \rangle
$$
and an equivalence of categories $\mathcal{A}_X \isom \Db(S)$.  More
generally, if $S$ is smooth, then the singular locus of $X$ is
contained in $L^\perp \cap G\dual$.
\end{prop}
\begin{proof}
  The smoothness statements, the semiorthogonal decomposition, and the
  equivalence of categories follow from Kuznetsov's theory of
  homological projective duality \cite{kuznetsov:hpd} applied to the
  duality between $G$ and a noncommutative resolution of
  $\Delta\dual$ that is supported along $G\dual$ (cf.\
  \cite[Thms.~4.1,~10.1,~10.4]{kuznetsov:HPD_grassmannians_lines} and
  also \cite[Thm.~3.1]{kuznetsov:cubic_fourfolds}).
\end{proof}

Building on work of Mukai~\cite{mukai:curves_grassmannians}, it is
proved in \cite[Corollary~6.4]{hoff_knutsen} that if $S=L \cap G$ is
smooth then
$L^\perp \cap G\dual \subset L^\perp \cap \Delta\dual = X$, which
contains the singular locus of $X$, is in bijection with the (finite)
set of elliptic pencils $E$ on $S$ of degree 5.

By Hassett~\cite{hassett:special}, any cubic fourfold of discriminant
14 has an associated K3 surface of degree 14.  To link pfaffian cubic
fourfolds and curves of genus 8 on the associated K3, we will need the
following.

\begin{prop}
\label{prop:pfaffian_BNgeneral}
A smooth cubic fourfold $X$ is pfaffian if and only if it has a discriminant
14 marking whose associated K3 surface $(S,H)$ is Brill--Noether
general.
\end{prop}
\begin{proof}
First suppose that $X=L^\perp \cap \Delta\dual$ is a smooth pfaffian
cubic fourfold. 
By Proposition~\ref{prop:HPD}, $S=L \cap G$ is a K3 surface of degree
14 with a polarization $H$ defined by the projective embedding $S \to
L \isom \P^{14}$ and there is an equivalence $\mathcal{A}_X \isom
\Db(S)$.  By
Mukai~\cite[Thms.~3.10]{mukai:curves_K3_surfaces_Fano_3-folds},
$(S,H)$ is Brill--Noether general.  By
Addington--Thomas~\cite{addington_thomas} (cf.\
\cite[Prop.~3.3]{huybrechts:K3_category_cubic_fourfold}), the
equivalence $\mathcal{A}_X \isom \Db(S)$ induces a Hodge isometry of
Mukai lattices $\widetilde{H}(\mathcal{A}_X,\Z) \isom
\widetilde{H}(S,\Z)$, which implies that $X$ has a marking of
discriminant 14 for which the associated polarized K3 surface is
$(S,H)$.

Now suppose that $X$ is a smooth cubic fourfold with a marking of
discriminant 14 whose associated polarized K3 surface $(S,H)$ is
Brill--Noether general.  Then by
Mukai~\cite[Prop.~4.7]{mukai:curves_K3_surfaces_Fano_3-folds}, $H$
defines a projective embedding whose image is
$S \isom L \cap G \subset L \isom \P^{14}$, for $L=\P(H^0(S,H)\dual)$
as described above.  Then
$X' = L^\perp \cap \Delta\dual \subset L^\perp \isom \P^5$ if a
pfaffian cubic fourfold whose singular points are in bijection with
elliptic pencils $E$ on $S$ such that $E.H=5$.  However, by
Remark~\ref{rem:quintic2}, $(S,H)$ cannot admit any such elliptic
pencils since it is associated to a smooth cubic fourfold of
discriminant 14.  Thus $X'$ must be a smooth pfaffian cubic fourfold,
so by Proposition~\ref{prop:HPD}, $\Db(X')$ has a semiorthogonal
decomposition
$\langle \mathcal{A}_{X'}, \OO_{X'}, \OO_{X'}(1), \OO_{X'}(2)\rangle$,
and there is an equivalence $\mathcal{A}_{X'} \isom \Db(S)$.  In
particular, by Addington--Thomas~\cite{addington_thomas}, $X'$ has a
marking of discriminant 14 whose associated polarized K3 surface is
$(S,H)$.  By the injectivity of
$\cC_{K_{14}} \hookrightarrow \Kthree_{14}$, we have that $X$ and $X'$
are isomorphic.  In particular, $X$ is pfaffian.
\end{proof}

In terms of the open immersion of moduli spaces,
Proposition~\ref{prop:pfaffian_BNgeneral} says that under the
open immersion $\cC_{K_{14}} \hookrightarrow \Kthree_{14}$, the pfaffian
locus coincides with the restriction of the Brill--Noether general
locus to the image.

For the very general cubic fourfold $X$ in $\cC_{14}$, the (unique)
associated polarized K3 surface $(S,H)$ of degree 14 has Picard rank
1.  The following is stated many times in the literature.

\begin{cor}
Any cubic fourfold $X$ of discriminant 14 and with $A(X)$ of rank 2 is pfaffian.
\end{cor}
\begin{proof}
Since the associated K3 surface $(S,H)$ has Picard rank 1, the smooth
curves $C \in |H|$ are Brill--Noether general, by
Lazarsfeld~\cite{lazarsfeld:Brill-Noether_without_degenerations}.
Hence Proposition~\ref{prop:pfaffian_BNgeneral} applies to show that $X$ is
pfaffian.
\end{proof}

As a further consequence, any cubic fourfold in $\cC_{14}\bslash \cPf$ has
$A(X)$ of rank $\geq 3$.

By Proposition~\ref{prop:pfaffian_BNgeneral}, $X$ is not pfaffian if
and only if $(S,H)$ is Brill--Noether special for every degree 14
marking on $X$.  Then by Theorem~\ref{thm:classify_K3}, we know that
$(S,H)$ must admit a lattice-polarization for some lattice in
Table~\ref{tab:BN-special_K3}.

The main result of this section is the following.

\begin{theorem}
\label{thm:clifford=3}
Let $X$ be a smooth cubic fourfold with a discriminant 14 marking and
$(S,H)$ an associated K3 surface of degree 14.  Then $\Cliff(S,H)=3$.
\end{theorem}
\begin{proof}
As noted in \S\ref{subsec:BN-K3}, if $(S,H)$ is Brill--Noether
general, then $\Cliff(S,H)=3$.  So we can assume that $(S,H)$ is
Brill--Noether special. In particular, $A(X)$ has rank $\geq 3$ and
let $\Lambda\subset A(X)$ be a primitive sublattice of rank 3
containing the marking $K_{14}$. Then by
Theorem~\ref{thm:classify_K3}, $(S,H)$ would admit an appropriate
$\sigma(\Lambda)$-polarization for $\sigma(\Lambda)$ given on
Table~\ref{tab:BN-special_K3}.  By
Proposition~\ref{prop:marked_embedding}, any such cubic fourfold $X$
would have a $\Lambda$-polarization, for the unique rank 3 cubic
fourfold lattice $\Lambda$ with specified $\sigma(\Lambda)$.  We have
enumerated these lattices $\Lambda$ in Table~\ref{tab:BN-special_K3}.
We now show that each such lattice $\Lambda$ corresponding to
$\Cliff(S,H) < 3$ has roots, hence no smooth $\Lambda$-polarized cubic
fourfolds exist in these cases.

When $\Cliff(S,H)=0$, the cubic fourfold lattice $\Lambda$ with
$(b,c)=(6,8)$ has short root $4h^2-3T+2J$.  Hence $A(X) \supset
\Lambda$ would contain a short root, thus no such smooth cubic
fourfold exists.

When $\Cliff(S,H)=1$, the cubic fourfold lattice $\Lambda$ with
$(b,c)=(5,6)$ has long root $4h^2-3T+3J$.  Hence $A(X) \supset
\Lambda$ would contain a long root, thus no such smooth cubic fourfold
exists.

When $\Cliff(S,H)=2$, we have two choices.  The cubic fourfold lattice
$\Lambda$ with $(b,c)=(2,2)$ has short root $J$.  Hence $A(X) \supset
\Lambda$ would contain a short root, thus no such smooth cubic
fourfold exists.  The cubic fourfold lattice $\Lambda$ with
$(b,c)=(4,4)$ has long root $4h^2-3T+3J$.  Hence $A(X) \supset
\Lambda$ would contain a long root, thus no such smooth cubic fourfold
exists.

This rules out all possibilities with $\Cliff(S,H) < 3$.  Hence
$(S,H)$ is Brill--Noether special with $\Cliff(S,H)=3$.
\end{proof}

\begin{remark}
\label{rem:quintic2}
This is a continuation of Remark~\ref{rem:g15}.  The rank 3 cubic
fourfold lattice~$\Lambda$, whose associated K3 surface lattice
$\sigma(\Lambda)$ has degree 14 generated by $H$ and~$E$ with $E.H=5$ and
$E^2=0$, must have $(b,c)=(1,2)$.  In this case, $\Lambda$ has short
root~$J$.  We conclude that there exist no smooth cubic fourfolds $X$
whose associated K3 surface has a line bundle restricting to a $g^1_5$
on the smooth genus 8 curves in the polarization class.  
\end{remark}

Given an admissible $d = 2g-2 > 6$, we wonder which values $0 \leq
\Cliff \leq {\lfloor (g-1)/2 \rfloor}$ can be realized by a Clifford
indices polarized K3 surfaces associated to Brill--Noether special
cubic fourfolds of discriminant $d$?  For example, we expect that
there are no ``hyperelliptic'' or ``trigonal'' special cubic fourfolds
of any admissible discriminant.

\section{Complement of the pfaffian locus}

In this section, we can finally prove Theorem~\ref{thm:main}.  We also
show that cubic fourfolds can admit multiple discriminant 14 markings,
some Brill--Noether special and some Brill--Noether general, implying
that the pfaffian locus is not open inside $\cC_{14}$.

\begin{proof}[Proof of Theorem~\ref{thm:main}]
As a consequence of Theorem~\ref{thm:clifford=3}, the only component
of the Brill--Noether special locus of $\cK_{14}$ that intersects the
image of $\cC_{K_{14}}$ is the one corresponding to $\gamma=3$ in
Table~\ref{tab:BN-special_K3}.  By Propositions~\ref{prop:Pi} and
\ref{prop:inPi}, we have that the intersection of this component with
the image of $\cC_{K_{14}}$ in $\cK_{14}$ coincides with the locus
$\cC_\Pi \subset \cC_{K_{14}}$ (where we always consider $\Pi$ with
the  fixed discriminant 14 marking derived from
\eqref{eq:Pi}).  Applying the forgetful map, we see that the
complement of the pfaffian locus in $\cC_{14}$ is contained in $\cC_{[\Pi]}$.
\end{proof}

\begin{example}
\label{notopen}
Consider the lattice $\Lambda$:
$$
\begin{array}{c|cccc}
\multicolumn{1}{c}{}   &h^2& T & P & P' \\\cline{2-5}
h^2 &3 &4 &1 &1 \\
T   &4 &10&0 &0 \\
P   &1 &0 &3 &0 \\ 
P'  &1 &0 &0 &3
\end{array}
$$
One can check, using the result of
Laza~\cite{laza:period_cubic_fourfold} and
Looijenga~\cite{looijenga:period_cubic_fourfold} on the image of the
period map, that $\Lambda$ is a cubic fourfold lattice, implying that
the locus $\cC_{[\Lambda]}$ of cubic fourfolds admitting a
$\Lambda$-polarized has codimension 3 in the moduli space $\cC$.  We
remark that the explicit example found by computers in
\cite[Ex.~A.2]{bolognesi_russo} is contained in $\cC_{[\Lambda]}$, and
motivated its definition.  By Proposition~\ref{prop:Pi}, the classes
$P$ and $P'$ correspond to disjoint planes, and the class $T$
generates a marking of discriminant 14, hence $\cC_\Lambda$ is a
divisor in $\cC_\Pi$.  Consider the four discriminant 14 markings
generated by the classes:
$$
T,\quad
T' = 2 h^2 - P - P',\quad
T'' = 3 h^2 - T - P,\quad
T''' = 3h^2 - T - P'.
$$
We compute that the polarized K3 surfaces $(S,H)$ of degree 14
associated to these markings admit lattice-polarizations for the
following lattices, respectively:
$$
\begin{array}{c|ccc}
\multicolumn{1}{c}{}   &H& C & C' \\\cline{2-4}
H &14&2 &2 \\
C &2 &-2&1 \\
C'&2 &1 &-2
\end{array}
\quad
\begin{array}{c|ccc}
\multicolumn{1}{c}{}   &H& L & C \\\cline{2-4}
H &14&7&4 \\
L &7 &2&2 \\
C &4 &2&-2
\end{array}
\quad
\begin{array}{c|ccc}
\multicolumn{1}{c}{}   &H& C & L \\\cline{2-4}
H &14&2 &1 \\
C &2 &-2&0 \\
L &1 &0 &-2
\end{array}
\quad
\begin{array}{c|ccc}
\multicolumn{1}{c}{}   &H& C & L \\\cline{2-4}
H &14&2 &1 \\
C &2 &-2&0 \\
L &1 &0 &-2
\end{array}
$$
Clearly $T''$ and $T'''$ are permuted up to reordering the planes $P$
and $P'$, so the last two lattice polarizations are isomorphic.

Now assume that $X$ is very general in $\cC_{[\Lambda]}$.  Then the
associated K3 surfaces have Picard lattices isomorphic to the ones
above and in all four cases, one can verify that the degree 14
polarization class is very ample by an analysis of the ample cone.  In
the first, third, and fourth cases, one can also verify using
Theorem~\ref{thm:classify_K3} that the smooth genus 8 curves in the
polarization class are Brill--Noether general, hence by
Proposition~\ref{prop:pfaffian_BNgeneral}, that $X$ is pfaffian (in
multiple ways).  In the second case, one can verify that the second
generator $C$ restricts to a $g^2_7$ on the smooth genus 8 curves in
the polarization class (hence they are Brill--Noether special by
Lemma~\ref{lem:BN-special_g27}).  By Theorem~\ref{thm:BNgeneral}, the
polarized K3 surface associated to this second marking is
Brill--Noether special.
\end{example}

\vspace{1.5cm}

\providecommand{\bysame}{\leavevmode\hbox to3em{\hrulefill}\thinspace}


\begin{thebibliography}{10}

\bibitem{addington_thomas}
Nicolas Addington and Richard Thomas, \emph{Hodge theory and derived categories of
  cubic fourfolds}, Duke Math. J. \textbf{163} (2014), no.~10, 1885--1927.

\bibitem{AHTVA}
Nicolas Addington, Brendan Hassett, Yuri Tschinkel, and Anthony
  V\'arilly-Alvarado, \emph{Cubic fourfolds fibered in sextic del pezzo
  surfaces}, Amer. J. Math. \textbf{141} (2019), no.~6, 1479–1500.

\bibitem{ABBV:pfaffian}
Asher Auel, Marcello Bernardara, Michele Bolognesi, and Anthony
  V{\'a}rilly-Alvarado, \emph{Cubic fourfolds containing a plane and a quintic
  del {P}ezzo surface}, Alg. Geom. \textbf{1} (2014), no.~2, 181--193.

\bibitem{beauville:determinantal_hypersurfaces}
Arnaud Beauville, \emph{Determinantal hypersurfaces}, Michigan Math. J.
  \textbf{48} (2000), 39--64, Dedicated to William Fulton on the occasion of
  his 60th birthday.

\bibitem{beauville_donagi}
Arnaud Beauville and Ron Donagi, \emph{La vari\'et\'e des droites d'une
  hypersurface cubique de dimension {$4$}}, C. R. Acad. Sci. Paris S\'er. I
  Math. \textbf{301} (1985), no.~14, 703--706.

\bibitem{bolognesi_russo}
Michele Bolognesi and Francesco Russo, \emph{Some loci of rational cubic
  fourfolds}, Math. Annalen \textbf{373} (2019), 165--190, with an appendix by
  Giovanni Staglian{\`o}.

\bibitem{ciliberto_pareschi}
Ciro Ciliberto and Giuseppe Pareschi, \emph{Pencils of minimal degree on curves
  on a {K3} surface}, J. Reine Angew. Math. \textbf{460} (1995), 15--36.

\bibitem{dolgachev:lattice_polarized_K3}
Igor V. Dolgachev, \emph{Mirror symmetry for lattice polarized {K3} surfaces},
  J. Math. Sci. \textbf{81} (1996), no.~3, 2599--2630, Translated from Itogi
  Nauki i Tekhniki, Seriya Sovremennaya Matematika i Ee Prilozheniya.
  Tematicheskie Obzory. Vol. 33, Algebraic Geometry-4, 1996.

\bibitem{donagi_morrison}
Ron Donagi and David~R. Morrison, \emph{Linear systems on {K3}-sections}, J.
  Differential Geom. \textbf{29} (1989), no.~1, 49--64.

\bibitem{eisenbud_harris:irreducibility}
David Eisenbud and Joe Harris, \emph{Irreducibility of some families of linear
  series with {B}rill-{N}oether number {$-1$}}, Ann. Sci. \'{E}cole Norm. Sup.
  (4) \textbf{22} (1989), no.~1, 33--53.

\bibitem{fano:cubic}
Gino Fano, \emph{Sulle forme cubiche dello spazio a cinque dimensioni
  contenenti rigate razionali del {$4^\circ$} ordine}, Comment. Math. Helv.
  \textbf{15} (1943), 71--80.

\bibitem{farkas_verra}
Gavril Farkas and Alessandro Verra, 
\emph{The universal {K3} surface of genus 14 via cubic fourfolds}, J. Math. Pures Appl \textbf{11} (2018), 1--20.

\bibitem{fulton:young_tableaux}
William Fulton, \emph{Young tableaux}, London Mathematical Society Student
  Texts, vol.~35, Cambridge University Press, Cambridge, 1997.

\bibitem{galluzzi}
Federica Galluzzi, \emph{Cubic fourfolds containing a plane and {K3} surfaces
  of picard rank two}, Geometriae Dedicata \textbf{186} (2017), no.~1,
  103--112.

\bibitem{green_lazarsfeld:divisors_curves_K3_surface}
Mark Green and Robert Lazarsfeld, \emph{Special divisors on curves on a {K3}
  surface}, Invent. Math. \textbf{89} (1987), no.~2, 357--370.

\bibitem{green:koszul_1}
Mark~L. Green, \emph{Koszul cohomology and the geometry of projective
  varieties}, J. Differential Geom. \textbf{19} (1984), no.~1, 125--171.

\bibitem{greer_li_tian}
 Francois Greer, Zhiyuan Li, Zhiyu Tian,
\emph{Picard groups on moduli of {K3} surfaces with {M}ukai models},
Int. Math. Res. Not. \textbf{16} (2015), 7238--7257.

\bibitem{hassett:some}
Brendan Hassett, \emph{Some rational cubic fourfolds}, J. Algebraic Geom.
  \textbf{8} (1999), no.~1, 103--114.

\bibitem{hassett:special}
\bysame, \emph{Special cubic fourfolds}, Compositio Math. \textbf{120} (2000),
  no.~1, 1--23.

\bibitem{hassett:cubic_survey}
Brendan Hassett, \emph{Cubic fourfolds, {K3} surfaces, and rationality
  questions}, Rationality Problems in Algebraic Geometry (R.~Pardini and G.P.
  Pirola, eds.), CIME Foundation Subseries, Lecture Notes in Mathematics, vol.
  2172, {S}pringer-{V}erlag, 2016, pp.~26--66.

\bibitem{hoff_knutsen}
Michael Hoff and  Andreas~Leopold Knutsen, \emph{{B}rill--{N}oether general {K3} surfaces with the maximal number of elliptic pencils of minimal degree}, Geom. Dedicata (2020).

\bibitem{huybrechts:K3_category_cubic_fourfold}
Daniel Huybrechts, \emph{The {K3} category of a cubic fourfold}, Compos. Math.
  \textbf{153} (2015), no.~3, 586--620.

\bibitem{huybrechts:K3_book}
\bysame, \emph{Lectures on {K}3 surfaces}, Cambridge Studies in Advanced
  Mathematics, vol. 158, Cambridge University Press, Cambridge, 2016.

\bibitem{johnsen_knutsen:K3}
Trygve Johnsen and Andreas~Leopold Knutsen, \emph{{K3} projective models in
  scrolls}, Lecture Notes in Mathematics, vol. 1842, {S}pringer-{V}erlag,
  Berlin, 2004.

\bibitem{knutsen:kth_order}
Andreas~Leopold Knutsen, \emph{On {$k$}th-order embeddings of {K3} surfaces and
  {E}nriques surfaces}, Manuscripta Math. \textbf{104} (2001), 211--237.

\bibitem{knutsen:smooth_curves_K3_surfaces}
\bysame, \emph{Smooth curves on projective {K3} surfaces}, Math. Scand.
  \textbf{90} (2002), no.~2, 215--231.

\bibitem{knutsen:gonality_Clifford_index}
\bysame, \emph{Gonality and {C}lifford index of curves on {K3} surfaces},
  Arch. Math. (Basel) \textbf{80} (2003), no.~3, 235--238.

\bibitem{knutsen:two_conjectures_curves_K3_surfaces}
\bysame, \emph{On two conjectures for curves on {K3} surfaces}, Internat. J.
  Math. \textbf{20} (2009), no.~12, 1547--1560.

\bibitem{kontsevich_tschinkel}
Maxim Kontsevich and Yuri Tschinkel, \emph{Specialization of birational types},
  Invent. Math. \textbf{217} (2019), no.~n, 415--432.

\bibitem{kuznetsov:HPD_grassmannians_lines}
Alexander Kuznetsov, \emph{Homological projective duality for {G}rassmannians
  of lines}, preprint arXiv:math/0610957, 2006.

\bibitem{kuznetsov:hpd}
\bysame, \emph{Homological projective duality}, Publ. Math. Inst. Hautes
  \'Etudes Sci. (2007), no.~105, 157--220.

\bibitem{kuznetsov:cubic_fourfolds}
\bysame, \emph{Derived categories of cubic fourfolds}, Cohomological and
  geometric approaches to rationality problems, Progr. Math., vol. 282,
  Birkh\"auser Boston Inc., Boston, MA, 2010, pp.~219--243.

\bibitem{kuznetsov:derived_rationality}
\bysame, \emph{Derived categories view on rationality problems}, Rationality
  problems in algebraic geometry, Lecture Notes in Math., vol. 2172, Springer,
  Cham, 2016, pp.~67--104.

\bibitem{laza:period_cubic_fourfold}
Radu Laza, \emph{The moduli space of cubic fourfolds via the period map}, Ann. of
  Math. \textbf{172} (2010), no.~1, 673--711.

\bibitem{lazarsfeld:Brill-Noether_without_degenerations}
Robert Lazarsfeld, \emph{Brill-{N}oether-{P}etri without degenerations}, J.
  Differential Geom. \textbf{23} (1986), no.~3, 299--307.

\bibitem{lelli-chiesa:stability_rank3}
Margherita Lelli-Chiesa, \emph{Stability of rank-3 {L}azarsfeld-{M}ukai bundles
  on {K3} surfaces}, Proc. Lon. Math. Soc. \textbf{107} (2013), no.~2,
  451--479.

\bibitem{lelli-chiesa:genealized_Lazarsfeld-Mukai_bundles}
\bysame, \emph{Generalized {L}azarsfeld-{M}ukai bundles and a conjecture of
  {D}onagi and {M}orrison}, Adv. Math. \textbf{268} (2015), no.~2, 529--563.

\bibitem{looijenga:period_cubic_fourfold}
Eduard Looijenga, \emph{The period map for cubic fourfolds}, Invent. Math.
  \textbf{177} (2009), 213--233.

\bibitem{martens:curves_on_K3_surfaces}
Gerriet Martens, \emph{On curves on {K3} surfaces}, Algebraic curves and projective
  geometry (1988), Lecture Notes in Mathematics, vol. 1398, Springer, 1989,
  pp.~174--182.

\bibitem{mukai:curves_K3_surfaces_Fano_3-folds}
Shigeru Mukai, \emph{Curves, {K3} surfaces and {F}ano $3$-folds of genus $\leq
  10$}, {A}lgebraic {G}eometry and {C}ommutative {A}lgebra: {I}n {H}onor of
  {M}asayoshi {N}agata (Tokyo), vol.~1, Kinokuniya, 1988, pp.~357--377.

\bibitem{mukai:curves_grassmannians}
\bysame, \emph{Curves and {G}rassmannians}, Algebraic geometry and related
  topics ({I}nchon, 1992), Conf. Proc. Lecture Notes Algebraic Geom., I, Int.
  Press, Cambridge, MA, 1993, pp.~19--40.

\bibitem{mukai:Fano_threefolds}
\bysame, \emph{New development of theory of {F}ano 3-folds: vector bundle
  method and moduli problem}, Sugaku \textbf{47} (1995), 125--144.

\bibitem{nikulin}
Viacheslav V. Nikulin, \emph{Integer symmetric bilinear forms and some of their
  geometric applications}, Izv. Akad. Nauk SSSR Ser. Mat. \textbf{43} (1979),
  no.~1, 111--177, 238.

\bibitem{reid:hyperelliptic_linear_systems_K3}
Miles Reid, \emph{Hyperelliptic linear systems on a {K}3 surface}, J. London
  Math. Soc. (2) \textbf{13} (1976), no.~3, 427--437.

\bibitem{reid:special_linear_systems_curves_K3}
\bysame, \emph{Special linear systems on curves lying on a {K}3 surface}, J.
  London Math. Soc. (2) \textbf{13} (1976), no.~3, 454--458.

\bibitem{russo_stagliano}
Francesco Russo and Giovanni Staglian{\`o}, \emph{Congruences of 5-secant
  conics and the rationality of some admissible cubic fourfolds}, Duke Math. J.
  \textbf{168} (2019), no.~5, 849--865.

\bibitem{russo_stagliano:42}
\bysame, \emph{Trisecant flops, their
  associated {K3} surfaces and the rationality of some fano fourfolds},
  arxiv:1909.01263.

\bibitem{saint-donat:projective_models_K3}
Bernard Saint-Donat, \emph{Projective models of {$K-3$} surfaces}, Amer. J. Math.
  \textbf{96} (1974), 602--639.

\bibitem{tregub:rationality_cubic_fourfold}
Semion L. Tregub, \emph{Three constructions of rationality of a cubic fourfold},
  Vestnik Moskov. Univ. Ser. I Mat. Mekh. (1984), no.~3, 8--14.

\bibitem{tregub:remarks_cubic_fourfolds}
\bysame, \emph{Two remarks on four-dimensional cubics}, Uspekhi Mat. Nauk
  \textbf{48} (1993), no.~2(290), 201--202.

\bibitem{voisin:cubic_fourfolds}
Claire Voisin, \emph{Th\'eor\`eme de {T}orelli pour les cubiques de {${\bf
  P}^5$}}, Invent. Math. \textbf{86} (1986), no.~3, 577--601.

\bibitem{voisin:aspects_Hodge_conjecture}
\bysame, \emph{Some aspects of the {H}odge conjecture}, Jpn. J. Math.
  \textbf{2} (2007), no.~2, 261--296.

\end{thebibliography}
\end{document}